\documentclass[12pt,a4paper]{article}
\usepackage{amstext}
\usepackage{fancyhdr}
\usepackage{amsfonts,graphicx,bezier, amssymb}
\usepackage{caption}
\usepackage{pgfcore}
\usepackage{amsmath,amsthm,amssymb,amsfonts,amscd, array}
\usepackage{default}
\newtheorem{defn}{Definition}[section]
 \newtheorem{prop}{Proposition}
  \newtheorem{lem}{Lemma}
  \newtheorem{thm}{Theorem}[section]
  \newtheorem{cor}[thm]{Corollary}
  \newcommand{\Z}{\mathbb Z}
  \newtheorem{exam}{Example}
\newenvironment{prf}{\noindent{\bf{Proof:}}~~}{\hfill\rule{1ex}{1ex}\vskip1.5ex}

\newcommand{\N}{\mathbb N}

\newcommand{\beqa}{\begin{eqnarray}}
\newcommand{\enqa}{\end{eqnarray}}
\newcommand{\beq}{\begin{eqnarray*}}
\newcommand{\enq}{\end{eqnarray*}}

\newtheorem{qn}{Question}[section]

\begin{document}

\begin{center}\Large{\bf{A complete radical formula  and  2-primal modules}}

\end{center}
\vspace*{0.3cm}
\begin{center}

David Ssevviiri\\

\vspace*{0.3cm}
Department of Mathematics\\
Makerere University, P.O BOX 7062, Kampala Uganda\\
E-mail: ssevviiri@cns.mak.ac.ug, ssevviirid@yahoo.com
\end{center}

\begin{abstract}

 We introduce a complete radical formula for modules over  non-commutative rings which is 
  the equivalence of a radical formula in the    setting of modules defined over commutative rings. 
  This   gives a   general frame work through which   known results about modules over commutative rings that
   satisfy the radical formula are retrieved. Examples and properties of modules that satisfy
  the complete radical formula  are given.  For instance, it is shown that a module that satisfies the complete radical formula is completely
  semiprime if and only if it is a subdirect product of completely prime modules. This generalizes a ring theoretical result: a ring is reduced
  if and only if it is a subdirect product of domains. We settle in affirmative a conjecture by Groenewald and the current author given
  in \cite{2p} that a module over    a 2-primal ring is 2-primal.
  More instances where 2-primal modules behave like modules over commutative rings are given. This is in tandem with the behaviour of 2-primal rings of 
  exhibiting tendencies of commutative rings.  We end with some questions about the role of 2-primal rings in algebraic geometry.  
  \end{abstract}

{\bf Keywords}: 2-primal rings, 2-primal modules, modules that satisfy the radical formula,  modules that satisfy the 
 complete radical formula.\\


\section{Introduction}

  A proper ideal $\mathcal{P}$ of a ring $R$ is a {\it prime} (resp.  {\it completely prime}) ideal if for all ideals $\mathcal{A}, \mathcal{B}$
  (resp. elements $a, b$) of $R$ such that 
  $\mathcal{AB}\subseteq \mathcal{P}$  (resp. $ab\in \mathcal{P}$), we have either $\mathcal{A}\subseteq \mathcal{P}$ or
  $\mathcal{B}\subseteq \mathcal{P}$ (resp. $a\in P$ or $b\in P$). The prime radical
  (resp. completely prime radical) of  a ring $R$ is the intersection of all prime (resp. completely prime) ideals of $R$.  
 Let $\mathcal{N}(R)$, $\beta(R)$ and $\beta_{co}(R)$ denote the set of all nilpotent elements of $R$, the prime radical of $R$ and 
the completely prime radical of $R$ (also called the generalized nilradical of $R$). Let $N$  be a submodule of an $R$-module $M$.
The envelope of a submodule $N$ of an $R$-module $M$ is the set $$E_M(N):=\{rm: r\in R, m\in M~\text{and}~r^km\in N~\text{for some}~k\in \N\}.$$
In general, $E_M(N)$ is not a
 submodule of $M$. We denote a submodule of $M$ generated by the envelope of $N$ by $\langle E_M(N)\rangle$.  The elements of $\langle E_M(0)\rangle$
 are called nilpotent elements of $M$.\\

If $R$ is a commutative ring  (or a 2-primal ring), then

\begin{equation}\label{eqn1}
  E_R(0) =\mathcal{N}(R)=\beta(R)=\beta_{co}(R).  
\end{equation}

 A desire to have   Equation  (\ref{eqn1}) (or parts of Equation (\ref{eqn1})) in the module setting, forms the basis of our study in this paper. 
 In literature, there are studies about the equivalences $\langle E_M(N)\rangle =\beta^s(N)$ and $\beta(M)=\beta_{co}(M)$ 
 in which  case one says the   submodule $N$ 
  of a module $M$  satisfies the radical formula and  a module $M$  is 2-primal respectively. $\beta^s(N)$ denotes the intersection
  of all prime submodules of a module  $M$ containing  a submodule $N$ and $\beta(M)$ (resp. $\beta_{co}(M)$) denotes the intersection of all prime (resp.
   completely prime) submodules of  a module $M$.   Whereas we supplement   studies of these two equivalences,  
   we also introduce   submodules that satisfy the complete radical formula, i.e., submodules $N$ of modules $M$ for which
  $\langle E_M(N)\rangle =\beta_{co}^s(N)$, where $\beta_{co}^s(N)$ is the intersection of all 
  completely prime submodules of a module $M$ containing a submodule $N$.
  This, naturally generalizes the notion of submodules that satisfy the radical formula in modules over commutative
   rings to modules over non-commutative rings. For an arbitrary ring, Levitzki showed that the set of all strongly nilpotent elements coincides with 
   the prime radical of that ring. We  give examples of modules that satisfy the module analogue of Levitzki result.  \\
   
  All rings are unital and  associative. The modules are left modules defined  over rings.

  \subsection*{Paper road map}
  
  This paper has seven sections.  We give the introduction in  Section 1. Preliminary results  which are needed later in the sequel are given in
  Sections 2 and 3.   Section 2 focuses on 2-primal rings and some of their properties, whereas in Section 3 we give relevant information about
  module analogues of    well known notions in ring theory. They include: prime modules, completely prime modules, modules that satisfy the radical
  formula and 2-primal modules.   It is in Section 4 that we introduce the complete radical formula of modules. As examples, it is shown that the 
  following modules satisfy    the complete radical formula: a projective and 2-primal module [Theorem \ref{TT1}], a finitely 
   generated module over a 2-primal ring [Theorem \ref{FG}], a completely prime module and the regular module $_RR$ when $R$ is a 2-primal ring
   [Theorem \ref{pl}].    In Corollaries \ref{11} and \ref{cv},  we have given several modules which are projective and 2-primal.  Furthermore, 
   all  the 
   modules given above   also satisfy both   the radical formula as well as  the module    analogue of Levitzki result for rings.  
  A ring $R$ satisfies a complete radical formula if every
  $R$-module satisfies the complete radical formula. It is shown that every semisimple 2-primal ring satisfies the complete radical formula [Corollary
  \ref{cc}]. We give a new characterization of 2-primal rings. A ring $R$ is 2-primal if and only if $\beta(R)=E_R(0)$ [Theorem \ref{prim}].
  In Section 5, we give an application of modules that satisfy the complete radical formula. If a module satisfies the 
  complete radical  formula, then it is completely semiprime if and only if it is a subdirect product of completely prime modules [Theorem \ref{rf}].
  In Section 6, we prove in affirmative a conjecture posed in \cite{2p}; it states that a module over a 2-primal ring  is 2-primal.  In Section 7, which
   is the last section, we give some information which inhibits the use of non-commutative rings in algebraic geometry. However, given the behaviour
   of 2-primal rings, i.e., having behavioural tendencies of commutative rings, we pose some questions on possible of using  2-primal rings  
    in algebraic geometry.

\section{2-primal rings}

\begin{defn}
 A ring $R$ is 2-primal if $$\mathcal{N}(R)=\beta(R).$$ 
\end{defn}

All commutative rings and all reduced rings are 2-primal.  The class of 2-primal rings has been widely studied, see for instance 
\cite{Gary, Marks, Marks2, Marks3, Shin} among others.

\begin{prop}\label{p}{\rm \cite[Proposition 2.1]{Gary}}
 Let $R$ be a ring, the following statements are equivalent:
 \begin{enumerate} 
  \item $R$ is 2-primal,
  \item $\beta_{co}(R)=\beta(R)$.
 \end{enumerate}

\end{prop}

An ideal $\mathcal{I}$ of a ring $R$ is {\it 2-primal} if    
\begin{equation}\label{l}
 \beta_{co}(R/\mathcal{I})=\beta(R/\mathcal{I}).
\end{equation}

 It follows that a ring is 2-primal if and only if its zero ideal is 2-primal.
 Equality  (\ref{l}) is the basis for the definition of 2-primal submodules, see Definition \ref{d1}.\\

The class of 2-primal rings is large. It contains many classes of generalizations of commutative rings:
 symmetric rings, IFP/SI rings, reversible rings, PSI rings, semi-symmetric rings, etc. For  examples and chart of implications
  among these classes, see \cite{Gary} and \cite{Marks}.\\

2-primal rings behave like commutative rings. For instance, just like commutative rings, they possess the following properties:
 
 \begin{enumerate}
  \item their sets of all nilpotent elements are ideals;
  \item they are Dedekind finite, i.e., if  $R$ is a 2-primal ring and $a, b\in R$ such that $ab=1$, then $ba=1$;  
  \item if   $R$ is a 2-primal ring, then the ring $R/\beta(R)$ is reduced, and hence it is IFP
  (i.e.,  if     $a, b\in R$, then $ab\in \beta(R)$ implies that $aRb\subseteq \beta(R)$),
  reversible (i.e.,  if      $a, b\in R$, then $ab\in \beta(R)$ implies that $ba\in \beta(R)$)
  and symmetric (i.e., if     $a, b, c\in R$, then $abc\in \beta(R)$ implies that $acb\in \beta(R)$);
  
  \item they satisfy  Kothe conjecture, i.e., the sum of two nil left ideals is nil;
  \item prime ideals are  completely prime and hence are strongly prime, strictly prime, $l$-prime and $s$-prime;
   
  \item they cannot be full matrix rings, \cite[p. 495]{Marks};
  \item the equality $$\mathcal{N}_s(R)=E_R(0) =\mathcal{N}(R)=\beta(R)=\beta_{co}(R)$$ holds, where 
  $\mathcal{N}_s(R)$ is the set of all strongly nilpotent 
   elements of $R$, see Corollary \ref{2m};
  \item semisimple 2-primal rings satisfy the radical formula, see Corollaries \ref{cb} and \ref{cc}, and Proposition \ref{da}.
  
\end{enumerate}
 
  \begin{defn}{\rm \cite[Definition 19]{Oliver}, \cite[p. 742]{Li}}
 A filtered ring $A$ is said to be almost commutative if the associated graded ring, $gr A=\bigoplus_{i\in I}(A_i/A_{i-1})$ is commutative.
\end{defn}
  
Basic examples of almost commutative rings involve rings of differential operators and  universal enveloping algebras.
\begin{exam}
 The almost commutative rings: the universal enveloping algebra of any Lie algebra over a field and 
 the ring of differential operators are reduced rings and hence 2-primal.
\end{exam}

We are then led to ask the following question:

\begin{qn}
 Is every  almost commutative ring 2-primal?
\end{qn}

\section{The module analogues}

In this section, we introduce module analogues of   prime rings,  completely prime rings (domains),  2-primal rings and modules that mimic 
the equivalence $\mathcal{N}(R)=\beta(R)$ in commutative rings.
 
 \subsection{Prime and completely prime modules}

\begin{defn}{\rm \cite{Dauns}}
 A submodule $P$ of an $R$-module $M$ is a prime submodule if $RM\not\subseteq P$ and for all ideals  $\mathcal{A}$ of $R$ and submodules
 $N$ of $M$ such that  $\mathcal{A}N\subseteq P$, we have $N\subseteq P$ or $\mathcal{A}M\subseteq P$.
\end{defn}

\begin{defn}{\rm \cite[Definition 2.1]{cp}}
 A submodule $P$ of an $R$-module $M$ is a completely prime submodule if $RM\not\subseteq P$ and for all elements  $r\in R$ and  
 $m\in M$ such that  $rm\in P$, we have $m\in P$ or $rM\subseteq P$.
\end{defn}

\begin{defn}
 A proper submodule $P$ of an $R$-module $M$ is a semiprime (resp. completely semiprime) submodule of $M$, if $RM\not\subseteq P$ and 
 for all $a\in R$ and $m\in M$, $aRam\subseteq P$   (resp. $a^2m\in P$) implies that $am\in P$.
\end{defn}

A module is prime (resp. completely prime, semiprime, completely semiprime) if its zero submodule is a prime (resp. completely prime, semiprime, 
completely semiprime) submodule.
The prime (resp. completely prime) radical of a submodule $N$ of an $R$-module $M$ 
is the intersection
of all  prime (resp. completely prime) submodules of $M$ containing $N$. We denote the prime (resp.  completely prime) radical of a nonzero submodule $N$
by  $\beta^s(N)$ (resp. $\beta_{co}^s(N)$).
Otherwise, if $N=0$, we write $\beta(M)$ (resp. $\beta_{co}(M)$) and call $\beta(M)$ (resp.  $\beta_{co}(M)$)
the  prime (resp. completely prime) radical of $M$. If $M$  has no prime (resp. completely prime)  submodules,
we write $\beta(M)=M$ (resp. $\beta_{co}(M)=M$).\\

A ring $R$ is prime (resp. completely prime, semiprime, completely semiprime)  if and only if the $R$-module $R$ is 
(resp. completely prime, semiprime, completely semiprime).
Any completely prime submodule is prime. Every maximal submodule is a prime submodule but it need not be  completely prime.
A torsion-free module, a simple module which is Lee-Zhou reduced  and a projective module over a domain are completely prime modules, see
\cite[Examples 2.2]{cp}, \cite[Example 2.3]{cp} and \cite[Example 3.10]{DS}. An indecomposable 
projective module over a hereditary Artin algebra is a completely prime module. To 
see this, if $M$ is an indecomposable projective module over a hereditary Artin algebra $R$, then by \cite[Proposition 5.1.1]{AR}
every nonzero map $f\in \text{End}_R(M)$ is a monomorphism. It follows from \cite[Proposition 3.1]{DS} that $M$ is a completely prime module.\\

Let $N$   be a submodule of  an $R$-module $M$, by $(N: M)$ we denote the ideal $$\{r\in R~:~ rM\subseteq N\}$$ of $R$ which is the annihilator
of the factor $R$-module $M/N$.

\begin{prop}
 A submodule $N$  of an $R$-module $M$ is completely prime  (resp. prime) if and only if $P=(N:M)$ is a completely prime  (resp. prime) ideal of $R$ and 
 the $R/P$-module $M/N$
 is torsion-free.
\end{prop}

\subsection{The radical formula of modules}

The equality $E_R(0)=\beta(R)$ 
from Equation (\ref{eqn1}) for commutative rings, motivated  McCasland and Moore in \cite{MM} to introduce (sub)modules 
that satisfy the radical formula.  On the other hand, the equality $\beta(R)=\beta_{co}(R)$ for 2-primal rings  motivated Groenewald and the current 
author in \cite{2p} to define 2-primal modules. In this subsection, we define, give examples and compare these  two  types of modules.
 
\begin{defn}
 A submodule $N$ of an $R$-module $M$ satisfies the radical formula if $$\langle E_M(N)\rangle=\beta^s(N).$$ A module $M$ satisfies the 
 radical formula if every submodule of $M$ satisfies the radical formula. A ring $R$ satisfies the radical formula if every $R$-module satisfies the radical
 formula.
\end{defn}

Modules  and rings that satisfy the radical formula have been widely studied, see \cite{Yilmaz, JS, LM, S, MM, NA, ST, SSM} among others.
A projective module over a commutative ring \cite[Corollary 8]{JS},  a module over a Dedekind integral
 domain  \cite[Theorem 9]{JS} and  a representable module (and hence an Artinian module) over a commutative ring
 satisfy the radical formula \cite[Theorem 9]{ST}.    A semisimple  commutative ring  and an Artinian commutative ring \cite{SSM}
    satisfy the radical formula.    Not all modules defined over commutative rings satisfy the radical formula.

\subsection{2-primal modules}

\begin{defn}\label{d1}
 A submodule $N$ of an $R$-module $M$ is 2-primal if $$\beta(M/N)=\beta_{co}(M/N).$$ A module $M$ is 2-primal if its zero submodule is 2-primal, i.e.,
 if $$\beta(M)=\beta_{co}(M).$$
\end{defn}

\begin{prop}{\rm \cite[Proposition 2.1]{2p}}
 A ring $R$ is 2-primal if and only if the module $_RR$ is 2-primal.
\end{prop}

\begin{exam}
 A Lee-Zhou reduced module (see \cite{LZ})
 and hence a completely prime module is 2-primal. Projective modules over 2-primal rings, 
 IFP modules, symmetric modules and modules over 
  commutative rings are 2-primal, see \cite{2p}.
\end{exam}

\begin{prop}
 If the prime radical of a module $M$ is a completely prime submodule of $M$, then $M$ is a 2-primal module.
\end{prop}

\begin{prf}
 Since for any completely prime submodule $P$ of $M$, $\beta(M)\subseteq \beta_{co}(M)\subseteq P$ when $\beta(M)$ is a completely prime 
 submodule of $M$, we get $\beta_{co}(M)\subseteq \beta(M)$ such that $\beta_{co}(M) = \beta(M)$.
\end{prf}

 We observe that   ``2-primal modules''  is a better generalisation than  ``modules that satisfy the radical formula''. 
   This is because, all modules over commutative rings are 2-primal just like all commutative rings are 2-primal.
   On the contrary, not all modules over commutative rings satisfy the radical formula. \\

   There was considerable effort   aimed at getting examples of modules that satisfy the radical formula.   
   Now that there is  a   generalisation   better than the notion of modules that satisfy the radical formula,
   i.e., that of 2-primal modules, it  is hoped that there will be interest by different researchers to  search  for more  examples of 2-primal modules, 
   in addition to those pointed out in this paper  and in  \cite{2p}.

\section{The complete radical formula}

 The inequality \begin{equation}\label{eqn0}
                                                  \beta_{co}(M)\subseteq \langle E_M(0)\rangle
                                                 \end{equation}
 which is equivalent to saying that 
 $\langle E_M(0)\rangle=\beta_{co}(M)$, (see Lemma \ref{1}) is a necessary and sufficient condition  
 for a zero submodule of a module  $M$ to satisfy the radical formula if  and only if
   $M$ is 2-primal, see \cite[Corollary 2.21]{RB}.\\

The motivation for studying (sub)modules that satisfy the complete radical formula is three fold. Firstly, it generalizes
 the notion of modules over commutative rings that satisfy the radical formula to modules over   not necessarily 
 commutative rings. Secondly,  it is a necessary and sufficient
 condition for modules to be 2-primal if and only if their zero submodules satisfy the radical formula.
Lastly, it allows every completely semiprime submodule to be an intersection of completely prime submodules which is not true in general.

\begin{defn}\label{def}
 Let $R$ be a ring and $M$   an $R$-module. A submodule $N$ of an $R$-module $M$   satisfies
 a complete radical formula if $$\langle E_M(N)\rangle=\beta_{co}^s(N).$$ A module satisfies the complete radical formula if every submodule
  of $M$ satisfies the complete radical formula. A ring $R$ satisfies the complete  radical formula if every $R$-module satisfies the complete radical
 formula.
\end{defn}

 \begin{lem}\label{1}{\rm \cite[Lemma 2.1]{RB}}   If $N$  is a submodule of an $R$-module $M$, then 
 $$\langle E_M(N)\rangle \subseteq \beta_{co}^s(N).$$
 \end{lem}

 \begin{cor}\label{cor1}
  If $I$  is a left ideal of a ring $R$, then $$\langle E_R(I)\rangle \subseteq \beta_{co}^s(I).$$
 \end{cor}
 
 We now give a new characterization of 2-primal rings.
 
 \begin{thm}\label{prim}
  A ring $R$ is 2-primal if and only if $$\beta(R)=E_R(0).$$
 \end{thm}

\begin{prf}
For any ring $R$, it is easy to see that $$\beta(R)\subseteq \mathcal{N}(R)\subseteq E_R(0)\subseteq \beta_{co}(R).$$
If $R$ is 2-primal, $\beta(R)=\beta_{co}(R)$ and hence $\beta(R)=E_R(0).$ For the converse, $\beta(R)=E_R(0)$  
implies that $\beta(R)=\mathcal{N}(R)$, which shows that $R$ is 2-primal.
 
\end{prf}

 \begin{cor}\label{2m}
  If $R$ is a 2-primal ring, then     $$\beta(R)=\mathcal{N}(R)= E_R(0)= \beta_{co}(R).$$    
    \end{cor}
 
 \begin{prf}
     Follows from the fact that  for any ring $R$, $\beta(R)\subseteq \mathcal{N}(R)\subseteq E_R(0)\subseteq \beta_{co}(R)$ and for 2-primal
    rings, $\beta(R)=\beta_{co}(R)$.    
 \end{prf}





 \begin{prop}\label{pr}
  If $N$ is a submodule of an $R$-module $M$, then  the following statements are equivalent:
  \begin{enumerate}
   \item  $E_M(N)=N$,
     \item $N$  is a completely semiprime submodule of $M$.
  \end{enumerate}
 \end{prop}
 
 \begin{prf}
  Elementary.
 \end{prf}

\begin{cor}\label{gd}
 A submodule $N$ of an $R$-module $M$  is  Lee-Zhou reduced if and only if it is both IFP and satisfies $E_M(N)=N$.
\end{cor}

Corollary \ref{gd} allows us to paraphrase Question 2.1 posed in paper \cite{RB} as:

\begin{qn}
 Is there a prime (resp. semiprime) module which is not completely prime (resp. Lee-Zhou reduced) but it is completely semiprime?
\end{qn}

A positive answer  to this question would lead to an example of a module which satisfies the radical formula but not 2-primal.\\

 An element $m$ of an $R$-module $M$ is   {\it strongly nilpotent}  \cite{Behboodi} if $m=\sum_{i=1}^ra_im_i$ for some $a_i\in R$, $m_i\in M$ 
 and $r\in \N$, such that for every $i$ $(1\leq i \leq r)$ and every sequence $a_{i1}, a_{i2}, a_{i3}, \cdots$ where
 $a_{i1}=a_i$ and $a_{in+1}\in a_{in}Ra_{in}$ (for all $n$), we have $a_{ik}Rm_i=0$ for some $k\in \N$. The set of all strongly nilpotent elements
 of a module $M$ is a submodule  and is denoted by $\mathcal{N}_s(M)$. 
 
 \begin{lem}\label{ll}
  For any $R$-module $M$, the following inequalities hold:
  
  $$ \mathcal{N}_s(M)\subseteq \langle E_M(0)\rangle \subseteq \beta_{co}(M).$$
 \end{lem}

 \begin{prf}
  Let $m\in \mathcal{N}_s(M)$, then $m=\sum_{i=1}^ra_im_i$ for some $a_i\in R$, $m_i\in M$ and $r\in \N$ such that for every $i$ ($1\leq i \leq r$)
  and every sequence  
  $a_{1i}, a_{2i}, a_{3i}, \cdots $ where $a_{i1}=a_i$ and $a_{n+1 i}\in a_{ni}Ra_{ni}$ for all $n$, we have $a_{ki}Rm_i=0$ for some $k\in \N$.
  For some $i$, choose the sequence 
  
  $$a_i, a_i^2, a_i^4, a_i^8, \cdots=\{a_i^{2^{r-1}}\}_{r=1}^{\infty}$$ then $a_{1i}=a_i$ and $a_{n+1 i}\in a_{ni}Ra_{ni}$ for all $n$. By hypothesis,
  there exists $k\in \N$ such that $a_{ki}Rm_i=0$. Since $a_{ki}=a_i^{2^{k-1}}$, it follows that
  $a_i^{2^{k-1}}m_i=0$. This implies $a_im_i\in E_M(0)$ so that   $m=\sum_{i=1}^ra_im_i\in \langle E_M(0)\rangle.$ The second inequality
  follows from Lemma \ref{1}.
 \end{prf}

If $R$ is a 2-primal ring, then we know that 

\begin{equation}\label{eqn2}
 \mathcal{N}(R)=\mathcal{N}_s(R)=E_R(0)=\beta_{co}(R)=\beta(R).
\end{equation}

For modules,  we have Theorem \ref{TM} below.

\begin{thm}\label{TM}
If $M$ is  a projective and 2-primal $R$-module, then
\begin{equation}\label{eq}
 \mathcal{N}_s(M)=\langle E_M(0)\rangle =\beta_{co}(M)=\beta(M).
\end{equation}  
Hence,  the zero submodule of $M$ satisfies both the complete radical formula  as well as the radical formula; and 
  $M$ satisfies the module analogue of Levitzki result for rings.
\end{thm}

\begin{prf}
If $M$ is  a projective and 2-primal module, then $\mathcal{N}_s(M) =\beta_{co}(M)=\beta(M)$, by \cite[Theorem 3.8]{Behboodi}
and the definition of 2-primal modules. Apply Lemma \ref{ll} to complete
the proof.
\end{prf}

\begin{cor}\label{11}
 For a projective module $M$ over any one of the following rings: 
    reduced rings,       commutative  rings,  left-duo rings,  symmetric rings, reversible rings, 
    IFP rings,  PSI rings,  semi-symmetric rings and 2-primal rings;    
the equality $$\mathcal{N}_s(M)=\langle E_M(0)\rangle =\beta_{co}(M)=\beta(M)$$ holds.
Hence,  the zero submodule of $M$ satisfies both the complete radical formula  as well as the radical formula; and 
  $M$ satisfies the module analogue of Levitzki result for rings.
\end{cor}

\begin{prf}
  Any of the above mentioned rings is 2-primal, see a chart of implications in \cite{Marks}. By \cite[Corollary 2.1]{2p}, a projective module over 
   a 2-primal ring is 2-primal. The rest follows from  Theorem \ref{TM}.
\end{prf}
 
 \begin{defn} \label{zee}
 An $R$-module $M$ is 
 \begin{enumerate}
  \item Lee-Zhou reduced \cite{LZ} if for all $a\in R$ and every $m\in M$, $am=0$ implies that $Rm \cap aM=0$. This is
  equivalent  to saying that:  for all $a\in R$ and every $m\in M$, $a^2m=0$ implies that $aRm=0$;  
 \item symmetric if  for $a, b\in R$  and $m\in M$, $abm=0$ implies that $bam=0$;
 \item semi-symmetric if for all $a\in R$ and  every $m\in M$,  $a^2m=0$ implies that $(a)^2m=0$ where
 $(a)$ is the ideal of $R$ generated by $a\in R$;
 \item  IFP (i.e., it has the insertion-of-factor-property) if whenever $am=0$ for $a\in R$ and $m\in M$,  we have $aRm=0$.
 \end{enumerate}
   \end{defn}
   

\begin{cor}\label{cv} For each of the following modules:
 \begin{enumerate}
\item $M$ is 2-primal and free,
\item $M$ is semi-symmetric and free,
 \item $M$ is  semi-symmetric and projective,
  \item $M$ is  IFP and projective,
 \item $M$ is  IFP and free,
 \item $M$ is symmetric and projective,
 \item $M$ is symmetric and free,
 \item $M$  is reduced and projective, 
 \item $M$ is reduced and free,
  \item $R$ is commutative and $M$ is projective,
 \item $R$ is commutative and $M$ is free;
 \end{enumerate}
   the equality $$\mathcal{N}_s(M)=\langle E_M(0)\rangle =\beta_{co}(M)=\beta(M)$$ holds.
Hence,  the zero submodule of $M$ satisfies both the complete radical formula  as well as the radical formula; and 
  $M$ satisfies the module analogue of Levitzki result for rings.
\end{cor}

 \begin{prf}
  By \cite[Theorems 2.2 and 2.3]{2p}, and the fact that every free module is projective, each of these modules is 2-primal and projective. The rest 
  follows from Theorem \ref{TM}.
 \end{prf}

Lemma \ref{th1} below can be proved  with appropriate modification of methods used to prove \cite[Theorem 1.5]{MM} for modules over commutative rings.

\begin{lem} \label{th1}
Let $\phi ~:~ M\rightarrow M'$ be an $R$-module epimorphism  and let  $N$ be  a submodule of $M$ such that $N\supseteq \text{Ker}~\phi$.
\begin{enumerate}
 \item[(i)] If $\beta_{co}^s(N)=\langle E_M(N)\rangle$, then $\beta_{co}^s(\phi (N))=\langle E_{M'}(\phi (N))\rangle$;
 \item[(ii)] If $N'$ is a submodule of $M'$ and $\beta_{co}^s(N')=\langle E_{M'}(N')\rangle$, then\\
 $\beta_{co}^s(\phi ^{-1}(N'))=\langle E_{M}(\phi ^{-1}(N'))\rangle$.
 \item[(iii)] If $\beta^s(N)=\langle E_M(N)\rangle$, then $\beta^s(\phi (N))=\langle E_{M'}(\phi (N))\rangle$;
 \item[(iv)] If $N'$ is a submodule of $M'$ and $\beta^s(N')=\langle E_{M'}(N')\rangle$, then\\
 $\beta^s(\phi ^{-1}(N'))=\langle E_{M}(\phi ^{-1}(N'))\rangle$.
\end{enumerate}
\end{lem}

\begin{thm}\label{TT1}
 If the $R$-module $M$ is any one of the modules given in Theorem \ref{TM} and Corollaries \ref{11} and \ref{cv},
 then $M$ satisfies both the complete radical   formula as well as the radical formula.
\end{thm}

\begin{proof}
 Let $N$ be a submodule of $M$. The modules $M$ given in Theorem \ref{TM} and Corollaries  \ref{11} and \ref{cv} 
 are  2-primal and projective. Hence, $\beta_{co}(M)=\langle E_M(0)\rangle$. When we apply Lemma \ref{th1}, by
 letting  $M'=M/N$ and $N'=N$,    we get $\beta_{co}^s(N)=\langle E_{M/N}(N)\rangle =\langle E_M(N)\rangle$, 
 i.e., every submodule of $M$ satisfies the complete radical  formula.   A similar argument starting with $\beta(M)=\langle E_M(0)\rangle$ shows
 that every submodule of $M$ satisfies the radical formula.
\end{proof}

Note that Theorem \ref{TT1} retrieves the well known result that a projective module over a commutative ring satisfies the radical formula, see 
\cite[Corollary 8]{JS}.

\begin{thm}\label{FG}
 A finitely generated (and  hence a cyclic) module over a 2-primal ring satisfies  both the radical formula as well as
 the complete radical formula. Hence, it
 is 2-primal, satisfies the module analogue of Levitzki result for rings and equality  (\ref{eq})  holds.
\end{thm}

\begin{prf} Let $R$ be a 2-primal ring. Since 2-primal rings are closed under direct sums (see \cite{Gary}),
the ring $R^n$ for some $n\in \N$ is also 2-primal.
 By Corollary \ref{2m},  $R^n$ considered as an $R$-module satisfies both the complete radical formula as well as the radical formula.
 By Lemma \ref{th1}, every homomorphic image of a module that satisfies the (complete) radical formula also satisfies the (complete) radical formula.
 Since a finitely generated  $R$-module is a homomorphic image of $R^n$,  it must also satisfy the (complete) radical formula.
\end{prf}

Theorem \ref{FG} retrieves an already known result: a  finitely generated module   
over a principal ideal domain  (resp. over a Dedekind domain) satisfies the radical formula, see \cite[Theorem 2]{MM} (resp. \cite[Theorem 9]{JS}).\\

For rings, every semiprime (resp. completely semiprime) ideal of $R$ is an intersection of prime (resp, completely prime) ideals. For modules,
this  is not true in general, see \cite[p. 3600]{JS}. However, for modules that satisfy the complete radical formula, we have Proposition \ref{pf}.

\begin{prop}\label{pf}
An $R$-module $M$ satisfies the complete radical formula if and only if every completely semiprime submodule $N$ of $M$ is an 
intersection of completely prime submodules of $M$.
\end{prop}

\begin{prf}
 Suppose $\langle E_M(N)\rangle =\beta_{co}^s(N)$ for every submodule $N$ of $M$. If $K$ is a completely semiprime submodule of $M$, then
 by Proposition \ref{pr}, $\langle E_M(K)\rangle =K$ such that by hypothesis, $\beta_{co}^s(K)=K$. This shows that $K$  is an intersection of 
  completely prime submodules of $M$. Conversely, suppose that  $K$  is an intersection of 
  completely prime submodules of $M$. Then $K$ is a completely semiprime submodule of $M$. By Proposition \ref{pr}, $\langle E_M(K)\rangle =K$.
  It follows that $K=\langle E_M(K)\rangle\subseteq \beta_{co}^s(K) \subseteq K$ and hence $\langle E_M(K)\rangle= \beta_{co}^s(K)$.
\end{prf}

The property of the module being 2-primal allows the module to behave as though it is defined over a commutative ring.
For modules over commutative rings, there is no distinction between modules that satisfy the radical formula    and those  that satisfy 
the complete radical formula. For 2-primal modules, we have Proposition \ref{ccc}.

 \begin{prop}\label{ccc}
  If $M$  is a 2-primal module, then the following statements are equivalent:
  \begin{enumerate}
   \item $M$ satisfies the complete radical formula,
   \item $M$ satisfies the radical formula.
     \end{enumerate}
 \end{prop}
 
 \begin{prf}
  If $\beta(M)=\beta_{co}(M)$, then $\langle E_M(0)\rangle =\beta(M)$ if and only if $\langle E_M(0)\rangle =\beta_{co}(M)$.
 \end{prf}

 \begin{defn}
  A ring $R$ is {\it left hereditary } if every submodule of a projective  $R$-module is projective.
 \end{defn}
 
 Semisimple rings, domains and  path algebras over a quiver are examples of left hereditary rings. A ring $R$  is semisimple if the regular
 module $_RR$ is a direct sum of simple submodules.

\begin{lem}\label{y}
 Let  $M$ be a projective module defined over a left hereditary ring $R$. Then  for any submodule $N$ of  $M$,
 $$\beta(N)\subseteq \langle E_N(0) \rangle \subseteq \beta_{co}(N).$$
\end{lem}

\begin{prf}
 Since $M$  is projective, so is every submodule $N$ of $M$ by definition of a left hereditary ring. $N$ projective, implies
 $\beta(N)=\beta(R)N$. Let $m\in \beta(N)$, then $m=\sum_{i=1}^ka_in_i$ where $a_i\in \beta(R)$,  $n_i\in N$ and $k\in \N$. Since $\beta(R)$ is nil, 
 $a_in_i\in E_N(0)\subseteq \langle E_N(0) \rangle$. This shows that $\beta(N)\subseteq \langle E_N(0) \rangle$. The second
  inequality follows from   Lemma \ref{1}.
\end{prf}

\begin{thm}\label{t}
 Let $M$ be a projective module over a hereditary ring $R$.
 If a submodule $N$ of $M$ is 2-primal considered as a module, then $N$ satisfies both  the radical formula as well as  the complete radical formula
 and hence $$\beta(N)=\langle E_N(0) \rangle= \beta_{co}(N).$$
  \end{thm}

\begin{prf}
 If $N$ is 2-primal (as a module), then $\beta(N)=\beta_{co}(N)$. Now apply Lemma \ref{y} to get $\beta(N)=\langle E_N(0) \rangle= \beta_{co}(N)$
 and Lemma \ref{th1} to see that $N$ satisfies both  the radical formula as well as  the complete radical formula.
\end{prf}

\begin{cor}\label{ca}
Let $M$ be a   module over a semisimple ring $R$.
   If a submodule $N$ of $M$ is 2-primal considered as a module, then $N$ satisfies both the radical formula as well as the complete radical formula.
\end{cor}

\begin{prf}
 If $R$ is a semisimple ring, then every $R$-module is projective. The rest follows from Theorem \ref{t}.
\end{prf}

\begin{cor}\label{cb}
 A semisimple commutative ring satisfies the radical formula.
\end{cor}

\begin{prf}
 Since every module over  a commutative ring $R$ is  2-primal, applying Corollary \ref{ca} when $R$ is semisimple shows that every submodule  
 of an $R$-module satisfies the radical formula and hence $R$ satisfies the radical formula.
\end{prf}

\begin{cor}\label{cc}
 A semisimple 2-primal ring satisfies  the complete radical formula.
\end{cor}

\begin{prf} A module $M$ over a semisimple 2-primal ring $R$ is projective and  2-primal by properties of semisimple rings and \cite[Theorem 1]{2p} respectively. 
 Since a semisimple ring is hereditary, every submodule $N$ of such a module is also projective. $N$ is  2-primal
 considered as a module by \cite[Theorem 1]{2p}. By Theorem \ref{t}, $N$ satisfies the complete radical formula.
 So,   $M$   satisfies the complete radical formula.
\end{prf}

Corollaries \ref{cb} and \ref{cc} give us another situation where 2-primal rings behave like commutative rings.\\

\begin{prop}\label{da}
 The following statements are equivalent:
 \begin{enumerate}
  \item a semisimple 2-primal ring satisfies the complete radical formula,
  \item a semisimple 2-primal ring satisfies the radical formula.
 \end{enumerate}

\end{prop}

 \begin{prf}
  Let $R$ be a semisimple 2-primal ring. Then any $R$-module $M$ is 2-primal. By Proposition \ref{ccc}, $M$ satisfies the complete radical formula
   if and only if it satisfies the radical formula. 
 \end{prf}

  Example \ref{exx} shows that it is possible for a submodule to satisfy the complete radical formula when it neither satisfies the radical formula nor
  2-primal.

 \begin{exam}\label{exx}
   Let $M=\left\{\begin{pmatrix}\bar{0} & \bar{0}\cr \bar{0} & \bar{0}\cr\end{pmatrix},
 \begin{pmatrix}\bar{0} &\bar{0}\cr \bar{1} & \bar{1}\cr\end{pmatrix},
 \begin{pmatrix}\bar{1} &\bar{1}\cr \bar{0} & \bar{0}\cr\end{pmatrix}, 
  \begin{pmatrix}\bar{1} &\bar{1}\cr \bar{1} & \bar{1}\cr\end{pmatrix}\right\}$ where entries of matrices in $M$ are from 
$\Z_2=\{\bar{0}, \bar{1}\}$ and $R=M_2(\Z)$.  The zero submodule of the $R$-module $M$ satisfies 
the complete radical formula, but $M$ is  neither 2-primal nor its zero submodule satisfies the radical formula.

 \end{exam}

\begin{prf}It suffices to  show that $0= \beta(M)\subsetneqq \beta_{co}(M)= \langle E_M(0)\rangle =M$.
Let $r=\begin{pmatrix}a & b \cr c &d\cr\end{pmatrix}\in R$, $$rM=\left\{\begin{pmatrix}\bar{0} & \bar{0}\cr \bar{0} & \bar{0}\cr\end{pmatrix},
\begin{pmatrix}a &a\cr c & c\cr\end{pmatrix}, \begin{pmatrix}b &b\cr d & d\cr\end{pmatrix}, 
 \begin{pmatrix}a+b &a+b\cr c+d & c+d\cr\end{pmatrix}\right\}\subseteq M$$ for any $a,b,c,d\in \Z$. 
 The would be non-trivial proper submodules, namely;\\ $N_1=\left\{\begin{pmatrix}\bar{0} & \bar{0}\cr \bar{0}
 & \bar{0}\cr\end{pmatrix},  \begin{pmatrix}\bar{1} &\bar{1}\cr \bar{0} & \bar{0}\cr\end{pmatrix}\right\}$, 
 $N_2=\left\{\begin{pmatrix}\bar{0} & \bar{0}\cr \bar{0} & \bar{0}\cr\end{pmatrix},  \begin{pmatrix}\bar{0} &\bar{0}\cr \bar{1}
 & \bar{1}\cr\end{pmatrix}\right\}$ and\\ $N_3=\left\{\begin{pmatrix}\bar{0} & \bar{0}\cr \bar{0} & \bar{0}\cr\end{pmatrix}, 
 \begin{pmatrix}\bar{1} &\bar{1}\cr \bar{1} & \bar{1}\cr\end{pmatrix}\right\}$ are not closed under
 multiplication by $R$ since, for $a$ and $c$ odd, $rN_1\not\subseteq N_1$,  for $b$ and $d$ odd, $rN_2\not\subseteq N_2$  and for $a$ odd but 
 $b, c, d$ even, 
$rN_3\not\subseteq N_3$.  This shows that $M$ is simple and hence prime. So, we have $\beta(M)=0$. However, if we take
$a=\begin{pmatrix}3 & 3 \cr2 &2\cr\end{pmatrix}\in R$ and 
$m=\begin{pmatrix}\bar{1} &\bar{1}\cr \bar{1} & \bar{1}\cr\end{pmatrix}\in M$, $am=0$ but $aM\not=0$ since 
$a=\begin{pmatrix}3 & 3 \cr2 &2\cr\end{pmatrix}\begin{pmatrix}\bar{1} & \bar{1} \cr\bar{0} &\bar{0}\cr\end{pmatrix}=
\begin{pmatrix}\bar{1} & \bar{1} \cr\bar{0} &\bar{0}\cr\end{pmatrix}\not=0$. This shows that $M$ is not completely prime. So, $M$ has no
 completely prime submodules, i.e., $\beta_{co}(M)=M$.  Note that 
 \begin{enumerate}
 \item $m_0=\begin{pmatrix} \bar{0} & \bar{0} \cr \bar{0} & \bar{0}\cr \end{pmatrix}=\begin{pmatrix} 1 & 0 \cr 0 & 1\cr \end{pmatrix}
 \begin{pmatrix} \bar{0} & \bar{0} \cr \bar{0} & \bar{0} \end{pmatrix}$ and 
 $\begin{pmatrix} 1 & 0 \cr 0 & 1\cr \end{pmatrix}^2
 \begin{pmatrix} \bar{0} & \bar{0} \cr \bar{0} & \bar{0} \end{pmatrix} =\begin{pmatrix} \bar{0} & \bar{0} \cr \bar{0} & \bar{0}\cr \end{pmatrix}$
  \item $ m_1= \begin{pmatrix}  \bar{1} & \bar{1}\cr  \bar{0}& \bar{0}\cr   \end{pmatrix}= \begin{pmatrix}  2 & 1\cr  2& 2\cr   \end{pmatrix}
 \begin{pmatrix}  \bar{0} & \bar{0}\cr  \bar{1}& \bar{1}\cr   \end{pmatrix}$ and 
 $   \begin{pmatrix}  2 & 1\cr  2& 2\cr   \end{pmatrix}^2
 \begin{pmatrix}  \bar{0} & \bar{0}\cr  \bar{1}& \bar{1}\cr   \end{pmatrix}=\begin{pmatrix}  \bar{0} & \bar{0}\cr  \bar{0}& \bar{0}\cr   \end{pmatrix}$
\item $ m_2= \begin{pmatrix}  \bar{0} & \bar{0}\cr  \bar{1}& \bar{1}\cr   \end{pmatrix}= \begin{pmatrix}  2 & 2\cr  1& 1\cr   \end{pmatrix}
 \begin{pmatrix}  \bar{1} & \bar{1}\cr  \bar{0}& \bar{0}\cr   \end{pmatrix}$ and 
 $   \begin{pmatrix}  2 & 2\cr  1 & 1\cr   \end{pmatrix}^2
 \begin{pmatrix}  \bar{1} & \bar{1}\cr  \bar{0}& \bar{0}\cr   \end{pmatrix}=\begin{pmatrix}  \bar{0} & \bar{0}\cr  \bar{0}& \bar{0}\cr   \end{pmatrix}$
\item $m_3=\begin{pmatrix}
            \bar{1}& \bar{1}\cr \bar {1} & \bar{1}\cr  \end{pmatrix}= \begin{pmatrix} 0 & 1 \cr 1 & 0 \cr  \end{pmatrix}
            \begin{pmatrix} \bar{0} & \bar{0} \cr \bar{1} & \bar{1}\cr  \end{pmatrix}$ a linear combination of $m_1$.
 
This shows  that $\langle E_M(0)\rangle =M$.

 \end{enumerate}
\end{prf}

\begin{thm}\label{pl}  The following modules satisfy the complete radical formula:
 \begin{enumerate}
   
  \item a module with $\beta_{co}(M)=0$, (e.g., when $M$ is completely  prime);
   \item a module with $\langle E_M(0) \rangle=M$ (e.g.,  a module given in Example \ref{exx});   
   \item the regular module $_RR$ when $R$ is 2-primal.
    
     \end{enumerate}
     Moreover, a module with $\beta_{co}(M)=0$ (and the regular module $_RR$ when $R$ is 2-primal)
     is 2-primal, satisfies both the radical formula and a module analogue of Levitzki result for rings.
\end{thm}

  \begin{prf} For each the three modules, $\beta_{co}(M)=\langle E_M(0)\rangle$. Now, apply Lemma \ref{th1} to get the desired result.
  For the second part, since $\beta(M)\subseteq \beta_{co}(M)=0$ and $\mathcal{N}_s(M)\subseteq \langle E_M(0)\rangle \subseteq \beta_{co}(M)=0$, we have 
  $\beta(M)=\mathcal{N}_s(M)=\langle E_M(0)\rangle= \beta_{co}(M)$. For the regular modules $_RR$, apply Corollary \ref{2m} with the fact that
  $\beta(R)=\beta(_RR)$,  $\beta_{co}(R)=\beta_{co}(_RR)$ and $\mathcal{N}_s(R)=\mathcal{N}_s(_RR)$.
     \end{prf}

\section{Application of modules that satisfy the complete radical formula}
 
 We know that a ring $R$ is reduced if and only if it is a subdirect product of domains. In general, this structure theorem is not true in the 
 module setting. This is due to the fact that not every completely semiprime submodule is an intersection of completely prime submodules.
 However, it holds when a module satisfies the complete radical formula, see Theorem \ref{rf}.

\begin{defn} 
A module $M$ is a subdirect product of the modules $S_{\lambda}$, $\lambda\in \Lambda$ if there is an injective
module homomorphism $\sigma : M\rightarrow S =\prod_{\lambda\in \Lambda} S_{\lambda}$ such that $\sigma \circ · \pi_{\lambda}$ is surjective for
all $\lambda\in \Lambda$ and for every canonical surjection $\pi_{\lambda}: S\rightarrow S_{\lambda}$.
\end{defn}

\begin{thm}\label{rf}
 Suppose that a module $M$ satisfies the complete radical formula, then the following statements are equivalent:
 \begin{enumerate}
  \item $M$ is completely semiprime,
  \item $\langle E_M(0)\rangle =0$,
  \item $\beta_{co}(M)=0$,
  \item $M$ is a subdirect product of  completely prime modules.
 \end{enumerate}

\end{thm}

 \begin{prf}
 
 \begin{enumerate}
  \item[$1\Leftrightarrow 2$.] Follows from Proposition \ref{pr}.
  \item[$ 2 \Leftrightarrow 3$.] Since $M$ satisfies the complete radical formula, $\langle E_M(0)\rangle =\beta_{co}(M)$. So, $\langle E_M(0)\rangle =0$ if 
          and only if $\beta_{co}(M)=0$.
  \item[$ 3\Rightarrow 4$.]  Suppose that $\beta_{co}(M)=0$. Let $\{N_{\lambda}\}_{\lambda\in \Lambda}$ be a collection of all completely prime 
  submodules of $M$. Then $\cap_{\lambda\in\Lambda}N_{\lambda}=0$ and $M$ is a subdirect product of modules $M/N_{\lambda}$, $\lambda\in \Lambda$
  which are completely prime. To see this, define $\sigma: M\rightarrow \prod_{\lambda\in \Lambda}(M/N_{\lambda})$ by 
  $\sigma(m)=(m+N_{\lambda})_{\lambda\in\Lambda}$. Then 
  $\text{Ker}~\sigma =\cap_{\lambda\in\Lambda}~\text{Ker}~\pi_{\lambda}=\cap_{\lambda\in\Lambda} N_{\lambda}$, $\sigma\circ \pi_{\lambda}$ is 
  surjective for every $\lambda\in\Lambda$ and $\sigma$ is injective if and only if $\cap_{\lambda\in\Lambda}N_{\lambda}=0$.
  \item[$ 4\Rightarrow 3$.]  Let $M$ be a subdirect product of  completely prime modules $\{S_{\lambda}\}_{\lambda\in \Lambda}$, i.e., there is an injection 
  $\sigma: M \rightarrow \prod_{\lambda\in\Lambda}S_{\lambda}$ with $\sigma\circ \pi_{\lambda}$ surjective, where
  $\pi_{\lambda}: \prod_{\lambda\in\Lambda}S_{\lambda}\rightarrow S_{\lambda}$    is the canonical surjection. Then, 
  $\text{Ker}(\sigma\circ \pi_{\lambda})$ is a completely prime submodule of $M$. Hence, 
  $\beta_{co}(M)\subseteq \cap_{\lambda\in\Lambda} \text{Ker}~(\sigma\circ \pi_{\lambda})=\text{Ker}~(\sigma)=0$ and $\beta_{co}(M)=0$.
 \end{enumerate}

 \end{prf}

 \begin{cor}\label{last}  Let $M$ be a module over a commutative ring. If $M$ satisfies the radical formula, then the following statements are equivalent:
 \begin{enumerate}
  \item $M$ is  semiprime,
  \item $\langle E_M(0)\rangle =0$,
  \item $\beta(M)=0$,
  \item $M$ is a subdirect product of  prime modules.
 \end{enumerate}

\end{cor}

\begin{prf}
 For modules defined over  commutative rings, prime (resp. semiprime) is indistinguishable from completely prime (resp. completely semiprime).
 Also, modules that satisfy the radical formula are indistinguishable from those that satisfy the complete radical formula.
\end{prf}

\begin{cor}
 If  a module $M$ satisfies the complete radical formula, 
 then the submodule $\beta_{co}(M)$ is the smallest  completely semiprime submodule of $M$.
\end{cor}

\begin{prf}
 If $M$ satisfies the complete radical formula, then by Theorem \ref{rf}, the zero submodule of $M$ is a completely semiprime submodule of $M$
 if and only if $\beta_{co}(M)=0$. In this case, there is no completely semiprime submodule of $M$ smaller than $\beta_{co}(M)$.
\end{prf}

\begin{cor} Suppose that $M$ is a module defined over a commutative ring.
 If  $M$ satisfies the radical formula,   then the submodule $\beta(M)$ is the smallest  semiprime submodule of $M$.
\end{cor}

\begin{prf}
 For modules over commutative rings, $\beta(M)=\beta_{co}(M)$ and for a module $M$ to satisfy the radical formula is equivalent to having $M$ satisfy
 the complete radical formula. The notion of semiprime is indistinguishable from that of completely semiprime.
\end{prf}

\begin{qn}
 Can we have Corollary \ref{last} for a module over a not necessarily commutative ring? i.e., a  module to be semiprime if and only if it
  is a subdirect product of prime modules. Note that a not necessarily commutative ring   is semiprime if and only if it is 
  a subdirect product of prime rings.
\end{qn}

Whereas we do not know the answer, we hasten to mention that $2\Leftrightarrow 3 \Leftrightarrow 4 \Rightarrow 1$ is easy to prove. So, the 
question reduces to  checking whether for modules $M$ that satisfy the radical formula, $M$ semiprime  implies $\beta(M)=0$. Note that,
this is not true in general, see Example \ref{exx}.

\section{A module over a 2-primal ring  is 2-primal}

  We have seen that a projective module over a  2-primal ring is 2-primal \cite[Corollary 2.1]{2p}, a finitely generated module
  (and hence a cyclic module) over 
   a 2-primal ring is 2-primal  [Theorem \ref{FG}] and a module over a commutative ring (a commutative ring is 2-primal) is 2-primal. This further
   compels ones belief in the  conjecture \cite[Conjecture 2.1]{2p} which states that   a  module defined over a 2-primal ring is  2-primal. 
   We show in Theorem  \ref{lT} that this conjecture is true.
   
   \begin{thm}\label{lT}
    A module over a 2-primal ring is  2-primal.
   \end{thm}
   \begin{prf}
    Let  $M$ be a module over a 2-primal ring $R$. We know that every module is  a homomorphic image of a projective module. So, there exists
    a projective $R$-module $P$ such that $M$  is a homomorphic image of $P$. By \cite[Corollary 2.1]{2p}, $P$ is a 2-primal module (since it is 
    a projective module over a 2-primal ring). To complete the proof, it is enough to show that every homomorphic image of a 2-primal module is 2-primal.
    But this is easy to see since for every $R$-module epimorphism $\phi : M\rightarrow M'$,  we have $\phi(\beta (M))=\beta(\phi(M))=\beta( M')$
    and $\phi(\beta_{co}(M))=\beta_{co}(\phi(M))=\beta_{co}(M')$.
   \end{prf}
   
   Theorem \ref{lT} shows that 2-primal modules are abundant.

 \section{Questions in algebraic geometry}
 
 Much of the algebraic geometry is done using commutative rings. Naturally, one wonders whether the algebraic geometry already known 
 for commutative rings can be developed for non-commutative rings. However,  
 there are two challenges  in trying to achieve this objective.
 1) Unlike commutative rings, non-commutative rings have fewer ideals and hence fewer prime ideals. As such, there is
 not usually a good topological space that reflects the ideal structure and representation theory of a given ring. 
 Hence, defining a projective scheme as a ringed topological space on the homogeneous primes of a ring would not be useful. 2)
 There isn't a   good theory of localization for non-commutative rings. So,  any attempt to
develop a non-commutative algebraic geometry based on rings and their localizations will not work, see  \cite{Keeler} and \cite{Smith}. \\

 Against this background together with the behaviour of 2-primal rings having tendencies of commutative rings,
 some questions come to mind.

 \begin{qn}
  Do 2-primal rings have as many ideals  (and hence as many prime ideals) as the commutative rings so that it is possible and useful
  to define a projective scheme as a ringed topological space on its homogeneous prime ideals?
 \end{qn}
 
 \begin{qn}
  Can one develop a good theory of localization for 2-primal rings? In other words, is the theory of localization 
   of 2-primal rings close to that of commutative rings
   that one can be able to do with 2-primal rings almost all that is done with commutative rings as regards localization?
 \end{qn}

 An affirmative answer to any one of the two questions above will increase on the class of rings for which certain algebraic geometry can be done.

\subsection*{Acknowledgement}
 The author acknowledges  support from  Sida   phase IV bilateral program with Makerere university, 2015--2020,  project: 316-2014 and wishes
  to thank Prof. Rikard Bogvad for the hospitality while at Stockholm university and for introducing him to almost commutative rings.








\end{document}